\documentclass[12pt,oneside]{amsart}
\usepackage{amssymb,verbatim,amscd,amsmath,graphicx}

\numberwithin{equation}{section}
\newtheorem{theorem}{Theorem}[section]
\newtheorem{lemma}[theorem]{Lemma}
\newtheorem{proposition}[theorem]{Proposition}
\newtheorem{corollary}[theorem]{Corollary}
\newtheorem{conjecture}[theorem]{Conjecture}

\theoremstyle{definition}
\newtheorem{definition}[theorem]{Definition}
\newtheorem{remark}[theorem]{Remark}
\newtheorem{example}[theorem]{Example}
\newtheorem{runningexample}[theorem]{Running Example}

\newtheorem{notation}[theorem]{Notation}
\newtheorem{question}[theorem]{Question}
\newtheorem{problem}[theorem]{Problem}

\DeclareMathOperator{\Ass}{Ass}

\DeclareMathOperator{\height}{ht}

\DeclareMathOperator{\gens}{gens}

\newcommand{\A}{\mathcal{A}}

\newcommand{\X}{\mathcal{X}}

\newcommand{\bm}{{\bf m}}

\newcommand{\ZZ}{{\mathbb Z}}

\newcommand{\RR}{{\mathbb R}}
\newcommand{\R}{{\mathcal R}}

\def\A{{\mathcal A}}

\def\R{{\mathcal R}}

\def\gr{\operatorname{gr}}

\def\X{{V}}

\def\c{{\bf c}}

\def\y{{\bf y}}
\def\z{{\bf z}}

\def\1{{\bf 1}}
\def\0{{\bf 0}}

\def\mba{{\bf a}}

\begin{document}


\title{Powers of squarefree monomial ideals and combinatorics}

\author{Christopher A. Francisco}
\address{Department of Mathematics, Oklahoma State University,
401 Mathematical Sciences, Stillwater, OK 74078}
\email{chris@math.okstate.edu}
\urladdr{http://www.math.okstate.edu/$\sim$chris}

\author{Huy T\`ai H\`a}
\address{Tulane University \\ Department of Mathematics \\
6823 St. Charles Ave. \\ New Orleans, LA 70118, USA}
\email{tha@tulane.edu}
\urladdr{http://www.math.tulane.edu/$\sim$tai/}

\author{Jeffrey Mermin}
\address{Department of Mathematics, Oklahoma State University,
401 Mathematical Sciences, Stillwater, OK 74078}
\email{mermin@math.okstate.edu}
\urladdr{http://www.math.okstate.edu/$\sim$mermin}

\thanks{This work was partially supported by grants from the Simons Foundation
(\#199124 to Francisco and \#202115 to Mermin). H\`a is partially supported
by NSA grant H98230-11-1-0165.}

\begin{abstract}
We survey research relating algebraic properties of powers of
squarefree monomial ideals to combinatorial structures. In particular, we
describe how to detect important
properties of (hyper)graphs by solving ideal membership problems and
computing associated primes. This work leads to algebraic
characterizations of perfect graphs independent of the Strong Perfect
Graph Theorem. In addition, we discuss the equivalence between the
Conforti-Cornu\'ejols conjecture from linear programming and the
question of when symbolic and ordinary powers of squarefree monomial ideals coincide.
\end{abstract}

\maketitle


\section{Introduction} \label{intro}

Powers of ideals are instrumental objects in commutative algebra. 
In addition, squarefree monomial ideals are intimately
connected to combinatorics. In this paper, we survey work on secant, symbolic
and ordinary powers of squarefree monomial ideals, and their combinatorial
consequences in (hyper)graph theory and linear integer programming.

There are two well-studied basic correspondences between squarefree monomial ideals and
combinatorics. Each arises from the identification of
squarefree monomials with sets of vertices of either a simplicial
complex or a hypergraph. The Stanley-Reisner correspondence associates to
the nonfaces of a simplicial complex $\Delta$ the generators of
a squarefree monomial ideal, and vice-versa. This framework leads to many
important results relating (mostly homological) ideal-theoretic properties of
the ideal to properties of the simplicial complex; see \cite[Chapter
5]{Bruns-Herzog} and \cite[Sections 61-64]{Peeva}.


The edge and cover ideal constructions identify the minimal generators of a squarefree
monomial ideal with the edges (covers) of a simple hypergraph. The edge ideal correspondence is
more na\"ively obvious but less natural than the Stanley-Reisner correspondence,
because the existence of a monomial in this ideal does not translate easily to
its presence as an edge of the (hyper)graph. Nevertheless, this correspondence has
proven effective at understanding properties of (hyper)graphs via algebra.
We focus on powers of squarefree monomial ideals when they are viewed as edge (or cover) ideals of hypergraphs. To the best of our
knowledge, there has been little systematic study of the powers of squarefree ideals from the Stanley-Reisner perspective.

The general theme of this paper is the relationship between symbolic
and ordinary powers of ideals. This topic has been investigated extensively in
the literature (cf. \cite{BH, ELS, HoHu, HKV}). Research along these lines has
revealed rich and deep interactions between the two types of powers of ideals,
and often their equality leads to interesting algebraic and geometric
consequences (cf. \cite{HHT, MT, TTerai, TT, Varbaro}). We shall
see that examining symbolic and ordinary powers of squarefree monomial ideals
also leads to exciting and important combinatorial applications.

The paper is organized as follows. In the next section, we collect notation and terminology. In Section \ref{s:properties}, we survey algebraic techniques for detecting important invariants and properties of (hyper)graphs. We consider three problems: 
\begin{enumerate}
\item computing the chromatic number of a hypergraph,
\item detecting the existence of odd cycles and odd holes in a graph, and
\item finding algebraic characterizations of bipartite and perfect graphs.
\end{enumerate}
We begin by describing two methods for determining the chromatic number of a hypergraph via an ideal-membership problem, one using secant ideals, and the other involving powers of the cover ideal. Additionally, we illustrate how the associated primes of the square of the cover ideal of a graph detect its odd induced cycles.

The results in Section \ref{s:properties} lead naturally to the investigation of associated primes of higher powers of the cover ideal. This is the subject of Section \ref{s:higherpowers}. We explain how to interpret the associated primes of the $s^{\text{th}}$ power of the cover ideal of a hypergraph in terms of coloring properties of its $s^{\text{th}}$ expansion hypergraph. Specializing to the case of graphs yields two algebraic characterizations of perfect graphs that are independent of the Strong Perfect Graph Theorem.

Section \ref{s:packing} is devoted to the study of when a squarefree monomial ideal has the property that its symbolic and ordinary powers are equal. Our focus is the connection between this property and the Conforti-Cornu\'ejols conjecture in linear integer programming. We state the conjecture in its original form and discuss an algebraic reformulation. This provides an algebraic approach for tackling this long-standing conjecture.

\emph{We congratulate David Eisenbud on his $65^{\text{th}}$ birthday, and this paper is written in his honor.}


\section{Preliminaries} \label{s:prel}

We begin by defining the central combinatorial object of the paper.

\begin{definition}
A \emph{hypergraph} is a pair $G=(V,E)$ where $V$ is a set, called the
\emph{vertices} of $G$, and $E$ is a subset of $2^{V}$, called the \emph{edges}
of $G$.  A hypergraph is \emph{simple} if no edge contains another; we allow the edges of a simple hypergraph to
contain only one vertex (i.e., isolated \emph{loops}). Simple hypergraphs have
also been studied under other names, including \emph{clutters} and \emph{Sperner
systems}. All hypergraphs in this paper will be simple.
\end{definition}

A \emph{graph} is a
hypergraph in which every edge has cardinality exactly two. We specialize to
graphs to examine special classes, such as cycles and perfect graphs. 

If $W$ is a subset of $V$, the \emph{induced sub-hypergraph} of $G$ on $W$ is
the pair $(W,E_{W})$ where $E_{W}=E\cap 2^{W}$ is the set of edges of $G$
containing only vertices in $W$.

\begin{notation} Throughout the paper, let $V=\{x_{1},\dots,x_{n}\}$ be a set
of vertices. Set $S=K[V]=K[x_{1},\dots,x_{n}]$, where $K$ is a field.  We will abuse notation by
identifying the squarefree monomial $x_{i_{1}}\dots x_{i_{s}}$ with the set
$\{x_{i_{1}},\dots, x_{i_{s}}\}$ of vertices.  If the monomial $m$
corresponds to an edge of $G$ in this way, we will denote the edge by $m$ as
well.
\end{notation}

\begin{definition} The \emph{edge ideal} of a hypergraph $G=(V,E)$ is
\[I(G)=(m:m\in E)\subset S.\]
On the other hand, given a
squarefree monomial ideal $I\subset S$, we let
$G(I)=(V,\gens(I))$ be the hypergraph associated to $I$, where $\gens(I)$ is
the unique set of minimal monomial generators of $I$.
\end{definition}

\begin{definition} A \emph{vertex cover} for a hypergraph $G$ is a set of
vertices $w$ such that every edge hits some vertex of $w$, i.e., $w\cap e\neq \varnothing$ for all edges $e$ of $G$.
\end{definition}

Observe that, if $w$ is a vertex cover, then appending a variable to $w$ results in another vertex cover.  In particular, abusing language slightly, the vertex covers form an ideal of $S$.

\begin{definition}  The \emph{cover ideal} of a hypergraph $G$ is
\[J(G)=(w:w \text{ is a vertex cover of }G).\]
\end{definition}

In practice, we compute cover ideals by taking advantage of duality.

\begin{definition}\label{d:alexdual}  Given a squarefree monomial ideal $I\subset S$, the \emph{Alexander dual} of $I$ is \[I^{\vee}=\bigcap_{m\in \gens(I)} \mathfrak{p}_{m},\] where $\mathfrak{p}_{m}=(x_{i}:x_{i}\in m)$ is the prime ideal generated by the variables of $m$.
\end{definition}

Observe that if $I=I(G)$ is a squarefree monomial ideal, its Alexander dual $I^\vee$ is also squarefree. We shall denote by $G^*$ the hypergraph corresponding to $I^\vee$, and call $G^*$ the \emph{dual hypergraph} of $G$. That is, $I^\vee = I(G^*)$. The edge ideal and cover ideal of a hypergraph are related by the following result.

\begin{proposition} \label{alexduality}
The edge ideal and cover ideal of a hypergraph are dual to each other:
$J(G)=I(G)^{\vee} = I(G^*)$ (and $I(G)=J(G)^{\vee}$). Moreover, minimal generators of
$J(G)$ correspond to minimal vertex covers of $G$, covers such that no proper
subset is also a cover.
\end{proposition}
\begin{proof}
Suppose $w$ is a cover.  Then for every edge $e$, $w\cap e\neq \varnothing$, so $w\in \mathfrak{p}_{e}$.  Conversely, suppose $w\in I(G)^{\vee}$.  Then, given any edge $e$, we have $w\in \mathfrak{p}_{e}$, i.e., $w\cap e\neq \varnothing$.  In particular, $w$ is a cover.
\end{proof}

We shall also need generalized Alexander duality for arbitrary monomial
ideals. We follow Miller and Sturmfels's book
\cite{MS}, which is a good reference for this topic. Let $\mathbf{a}$ and $\mathbf{b}$ be
vectors in $\mathbb{N}^n$ such that $b_i \le a_i$ for each $i$. As in
\cite[Definition 5.20]{MS},
we define the vector $\mathbf{a \setminus b}$ to be the vector whose $i^{\text{th}}$
entry is given by
\[
a_i \setminus b_i = \left\{
\begin{array}{l l}
 a_i+1-b_i & \text{if } b_i \ge 1 \\
 0 & \text{if } b_i=0. \\
\end{array}
\right.
\]

\begin{definition} \label{def.generalalexanderdual}
Let $\mathbf{a} \in \mathbb{N}^n$, and let $I$ be a monomial ideal such that all
the minimal generators of $I$ divide $\mathbf{x}^\mathbf{a}$.
The \emph{Alexander dual} of $I$ \emph{with respect to} $\mathbf{a}$ is the
ideal
\[ I^{[\mathbf{a}]} = \bigcap_{\mathbf{x}^\mathbf{b} \in \gens(I)} \,
(x_1^{a_1 \setminus b_1}, \dots, x_n^{a_n \setminus b_n}).\]
\end{definition}

For squarefree monomial ideals, one obtains the usual Alexander dual by taking
$\mathbf{a}$ equal to $\mathbf{1}$, the vector with all entries 1, in Definition \ref{def.generalalexanderdual}.

By Definition \ref{d:alexdual}, Alexander duality identifies the minimal generators of a squarefree ideal with the primes associated to its dual.  The analogy for generalized Alexander duality identifies the minimal generators of a monomial ideal with the \emph{irreducible components} of its dual.

\begin{definition} A monomial ideal $I$ is \emph{irreducible} if it has the form $I=(x_{1}^{e_{1}},\dots, x_{n}^{e_{n}})$ for $e_{i}\in \ZZ_{>0}\cup\{\infty\}$.  (We use the convention that $x_{i}^{\infty}=0$.)  Observe that the irreducible ideal $I$ is $\mathfrak{p}$-primary, where $\mathfrak{p}=(x_{i}:e_{i}\neq \infty)$.
\end{definition}
\begin{definition} Let $I$ be a monomial ideal.  An \emph{irreducible decomposition} of $I$ is an irredundant decomposition
\[
I=\bigcap Q_{j}
\]
with the $Q_{j}$ irreducible ideals.  We call these $Q_{j}$  \emph{irreducible components} of $I$.  By Corollary \ref{c:irreddecomp} below, there is no choice of decomposition, so the irreducible components are an invariant of the ideal.
\end{definition}
\begin{proposition} Let $I$ be a monomial ideal, and $\mathbf{a}$ be a vector with entries large enough that all the minimal generators of $I$ divide $\mathbf{x}^{\mathbf{a}}$.  Then $(I^{[\mathbf{a}]})^{[\mathbf{a}]}=I$.  
\end{proposition}
\begin{corollary}\label{c:irreddecomp}  Every monomial ideal has a unique irreducible decomposition.
\end{corollary}

A recurring idea in our paper is the difference between the powers and symbolic powers of squarefree ideals.  We recall the definition of the symbolic power.

For a squarefree monomial ideal $I$, the \emph{$s^{\text{th}}$ symbolic power} of $I$ is
\[I^{(s)}=\bigcap_{\mathfrak{p}\in\Ass(S/I)}\mathfrak{p}^{s}.\]
(This definition works because squarefree monomial ideals are the intersection of prime ideals.
For general ideals (even general monomial
ideals) the definition is more complicated.)
In general we have $I^{s} \subseteq I^{(s)}$, but the precise nature of
the relationship between the symbolic and ordinary powers of an ideal
is a very active area of research.

In commutative algebra, symbolic and ordinary powers of an ideal are encoded in the symbolic Rees algebra and the ordinary Rees algebra. More specifically, for any ideal $I \subseteq S = K[x_1, \dots, x_n]$, the \emph{Rees algebra} and the \emph{symbolic Rees algebra} of $I$ are
$$\R(I) = \bigoplus_{q \ge 0}I^q t^q \subseteq S[t] \text{ and } \R_s(I) = \bigoplus_{q \ge 0} I^{(q)} t^q \subseteq S[t].$$
The symbolic Rees algebra is closely related to the Rees algebra, but often is richer and more subtle to understand. For instance, while the Rees algebra of a homogeneous ideal is always Noetherian and finitely generated, the symbolic Rees algebra is not necessarily Noetherian. In fact, non-Noetherian symbolic Rees algebras were used to provide counterexamples to Hilbert's Fourteenth Problem (cf. \cite{Nagata, Roberts}). 


\section{Chromatic number and odd cycles in graphs} \label{s:properties}

In this section, we examine how to detect simple graph-theoretic
properties of a hypergraph $G$ from (powers of) its edge and cover ideals. Since the results in this section involving chromatic
number are the same for graphs as for hypergraphs, modulo some
essentially content-free extra notation, we encourage novice readers
to ignore the hypergraph case and think of $G$ as a graph. 



\begin{definition}  Let $k$ be a positive integer. A \emph{$k$-coloring} of $G$ is an
  assignment of colors $c_{1},\dots, c_{k}$ to the vertices of $G$ in
  such a way that every edge of cardinality at least 2 contains vertices with different colors.
  We say that $G$ is \emph{$k$-colorable} if a $k$-coloring of $G$
  exists, and that the \emph{chromatic number} $\chi(G)$ of $G$ is the least $k$
  such that $G$ is $k$-colorable.
\end{definition}

\begin{remark} Since loops do not contain two vertices, they cannot contain two vertices of different colors.  Thus the definition above considers only edges with cardinality at least two.  Furthermore, since the presence or absence of loops has no effect on the chromatic number of the graph, we will assume throughout this section that all edges have cardinality at least two.
\end{remark}

\begin{remark} For hypergraphs, some texts instead define a
  coloring of $G$ to be an assignment of colors to the vertices such
  that no edge contains two vertices of the same color.  However, this
  is equivalent to a coloring of the one-skeleton of $G$, so the
  definition above allows us to address a broader class of problems.
\end{remark}


\begin{runningexample} \label{running} Let $G$ be the graph obtained by gluing a pentagon to a square along one edge, shown in Figure \ref{f:examplefigure}.  The edge ideal of $G$ is $I(G)=(ab,bc,cd,de,ae,ef,fg,dg)$.  The chromatic number of $G$ is 3:  for example, we may color vertices $a$, $c$, and $g$ red, vertices $b$, $d$, and $f$ yellow, and vertex $e$ blue.
\end{runningexample}
\begin{figure}[hbtf]
\includegraphics[height=1.25in]{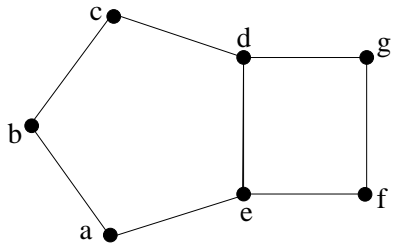}
\caption{The graph $G$ in the running example}\label{f:examplefigure}
\end{figure}

The chromatic number of $G$ can be determined from the solutions to
either of two different ideal membership problems.

Observe that a graph fails to be $k$-colorable if and only if every
assignment of colors to its vertices yields at least one
single-colored edge.  Thus, it suffices to test every color-assignment
simultaneously.  To that end, let $Y_{1},\dots, Y_{k}$ be distinct
copies of the vertices:  $Y_{i}=\{y_{i,1},\dots, y_{i,n}\}$.  We think
of $Y_{i}$ as the $i^{\text{th}}$ color, and the vertices of $Y_{i}$
as being colored with this color.  Now let $I(Y_{i})$ be the edge
ideal $I = I(G)$, but in the variables $Y_{i}$ instead of $V$.  Now an
assignment of
colors to $G$ corresponds to a choice, for each vertex $x_{j}$, of a
colored vertex $y_{i,j}$; or, equivalently, a monomial of the form
$y_{i_{1},1}y_{i_{2},2}\dots y_{i_{n},n}$.  This monomial is a
coloring if and only if it is not contained in the the monomial ideal
$\widetilde{I}=I(Y_{1})+\dots + I(Y_{k})$.  In particular,
$G$ is $k$-colorable if and only if the sum of all
such monomials is not contained in $\widetilde{I}$.

We need some more notation to make the preceding discussion into a
clean statement.  Let $\mathbf{m}=x_{1}\dots x_{n}$, let
$T_{k}=K[Y_{1},\dots,Y_{k}]$, and let $\phi_{k}:S\to T_{k}$ be the
homomorphism sending $x_{i}$ to $y_{1,i}+\dots + y_{k,i}$.  Then
$\phi_{k}(\mathbf{m})$ is the sum of all color-assignments, and
we have shown the following:

\begin{lemma}\label{l:secantcoloringlemma}
With notation as above, $G$ is $k$-colorable if and only if
$\phi_{k}(\mathbf{m})\not\in\widetilde{I}$.
\end{lemma}

We recall the definition of the $k^{\text{th}}$ secant ideal. Secant
varieties are common in algebraic geometry, including in many recent
papers of Catalisano, Geramita, and Gimigliano (e.g., \cite{CGG}),
and, as Sturmfels and Sullivant note in \cite{SS}, are playing an important role in algebraic statistics.

\begin{definition}
Let $I\subset S$ be any ideal, and continue to use all the notation above.
Put $T=K[V,Y_{1},\dots, Y_{k}]$ and regard $S$ and $T_{k}$ as subrings
of $T$.  Then the \emph{$k^{\text{th}}$ secant power} of $I$ is
\[
I^{\{k\}}=S\cap\left(\widetilde{I}+\left(\{x_{i}-\phi_{k}(x_{i})\}\right)\right).
\]
\end{definition}

Lemma \ref{l:secantcoloringlemma} becomes the following theorem of Sturmfels and Sullivant \cite{SS}:

\begin{theorem} \label{t:secantcoloring}
$G$ is $k$-colorable if and only if $\mathbf{m}\not\in I(G)^{\{k\}}$. In particular,
$$\chi(G) = \min \{ k ~|~ \mathbf{m} \not\in I(G)^{\{k\}} \}.$$
\end{theorem}

\begin{runningexample}
Let $G$ and $I$ be as in Example \ref{running}.  Then $I^{\{1\}}=I$ and $I^{\{2\}}=(abcde)$ both contain the monomial $abcdefg$.  However, $I^{\{3\}}=0$.  Thus $G$ is 3-colorable but not 2-colorable.
\end{runningexample}

Alternatively, we can characterize chromatic number by looking
directly at powers of the cover ideal.

Observe that, given a $k$-coloring of $G$, the set of vertices which
are not colored with any one fixed color forms a vertex cover of $G$.
In particular, a $k$-coloring yields $k$ different vertex covers, with
each vertex missing from exactly one.  That is, if we denote these
vertex covers $w_{1},\dots, w_{k}$, we have $w_{1}\dots
w_{k}=\mathbf{m}^{k-1}$.  In particular, we have the following result of Francisco, H\`a, and Van Tuyl \cite{FHVTperfect}.

\begin{theorem} \label{t:covercoloring}
$G$ is $k$-colorable if and only if $\mathbf{m}^{k-1}\in J(G)^{k}$. In particular,
$$\chi(G) = \min \{k ~|~ \mathbf{m}^{k-1} \in J(G)^k \}.$$
\end{theorem}

\begin{proof}  Let $J = J(G)$. Given a $k$-coloring, let $w_{i}$ be the set of
  vertices assigned a color other than $i$.  Then
  $\mathbf{m}^{k-1}=w_{1}\dots w_{k}\in J^{k}$.  Conversely, if
  $\mathbf{m}^{k-1}\in J^{k}$, we may write
  $\mathbf{m}^{k-1}=w_{1}\dots w_{k}$ with each $w_{i}$ a squarefree
  monomial in $J$. Assigning the color $i$ to the complement of
  $w_{i}$ yields a $k$-coloring: indeed, we have $\prod \frac{\mathbf{m}}{w_i} = \frac{\mathbf{m}^k}{\mathbf{m}^{k-1}}=\mathbf{m}$, so the $\frac{\mathbf{m}}{w_i}$ partition $V$.
\end{proof}

\begin{runningexample}
In Example \ref{running}, let $\mathbf{m}=abcdefg$.
The cover ideal $J(G)$ is $(abdf,acdf,bdef,aceg,bceg,bdeg)$.  Because $J$ does not contain $\mathbf{m}^{0}=1$, $G$ is not $1$-colorable.  All 21 generators of $J^{2}$ are divisible by the square of a variable, so $G$ is not $2$-colorable.  Thus $\mathbf{m}\not\in J^{2}$, so $J$ is not $2$-colorable.  However, $J^{3}$ contains $\mathbf{m}^{2}$, so $G$ is $3$-colorable.
\end{runningexample}

\begin{remark} \label{r.bfold}
One can adapt the proof of Theorem~\ref{t:covercoloring} to determine the $b$-fold chromatic number of a graph, the minimum number of colors required when each vertex is assigned $b$ colors, and adjacent vertices must have disjoint color sets. See \cite[Theorem 3.6]{FHVTperfect}.
\end{remark}

\begin{remark} The ideal membership problems in Theorems \ref{t:secantcoloring} and \ref{t:covercoloring} are for monomial ideals, and so they are computationally simple. On the other hand, computing the chromatic number is an \textbf{NP}-complete problem. The bottleneck in the algebraic algorithms derived from Theorems \ref{t:secantcoloring} and \ref{t:covercoloring} is the computation of the secant ideal $I(G)^{\{k\}}$ or the cover ideal $J(G)$ given $G$; these problems are both \textbf{NP}-complete.
\end{remark}

It is naturally interesting to investigate the following problem.

\begin{problem} Find algebraic algorithms to compute the chromatic number $\chi(G)$ based on algebraic invariants and properties of the edge ideal $I(G)$.
\end{problem}

For the rest of this section, we shall restrict our attention to the case when $G$ is a graph (i.e., not a hypergraph), and consider the problem of identifying odd cycles and odd holes in $G$. As before, let $I = I(G)$ and $J = J(G)$.

Recall that a \emph{bipartite graph} is a two-colorable graph, or,
equivalently, a graph with no odd circuits.  This yields two corollaries
to Theorem \ref{t:covercoloring}:

\begin{corollary}\label{c:bipartite}
$G$ is a bipartite graph if and only if $\mathbf{m}\in J^{2}$.
\end{corollary}
\begin{corollary}\label{c:oddcycle}
If $G$ is a graph, then $G$ contains an odd circuit if and only if
$\mathbf{m}\not\in J^{2}$.
\end{corollary}


It is natural to ask if we can locate the offending odd circuits.
In fact, we can identify the \emph{induced odd cycles} from the
associated primes of $J^{2}$.

\begin{definition}  Let $C=(x_{i_{1}},\dots, x_{i_{s}}, x_{i_{1}})$ be a circuit in
  $G$.  We say that $C$ is an \emph{induced cycle} if the induced
  subgraph of $G$ on $W=\{x_{i_{1}},\dots, x_{i_{s}}\}$ has no edges
  except those connecting consecutive vertices of $C$.  Equivalently,
  $C$ is an induced cycle if it has no chords.
\end{definition}

\begin{runningexample}
$G$ has induced cycles $abcde$ and $defg$.  The circuit $abcdgfe$ isn't an induced cycle, since it has the chord $de$.
\end{runningexample}

Simis and Ulrich prove that the odd induced cycles are the generators of the second secant ideal of $I$ \cite{SU}.
\begin{theorem}\label{simisulrich}
Let $G$ be a graph with edge ideal $I$.  Then a squarefree monomial $m$ is a generator of $I^{\{2\}}$ if and only if $G_{m}$ is an odd induced cycle.   
\end{theorem}
\begin{proof}[Sketch of proof.] If $G_{m}$ is an odd induced cycle, then $G_{m}$ and hence $G$ are not $2$-colorable.  On the other hand, if $m\in I^{\{2\}}$, then $G_{m}$ is not $2$-colorable and so has an odd induced cycle.
\end{proof}

Now suppose that $G$ is a cycle on $(2\ell -1)$ vertices, so without
loss of generality $I=(x_{1}x_{2}, x_{2}x_{3}, \dots,
x_{2\ell-1}x_{1})$. Then the generators of $J$ include the $(2\ell -1)$ vertex
covers $w_{i}=x_{i}x_{i+2}x_{i+4}\dots x_{i+2\ell -2}$ obtained by
starting anywhere in the cycle and taking every second vertex until we
wrap around to an adjacent vertex.  (Here we have taken the subscripts
mod ($2\ell -1$) for notational sanity.)  All other generators have
higher degree.  In particular, the
generators of $J$ all have degree at least $\ell$, so the generators
of $J^{2}$ have degree at least $2\ell$.  Thus $\mathbf{m}\not\in
J^{2}$, since $\deg(\mathbf{m})=2\ell -1$.  However, we have
$\mathbf{m}x_{i}=w_{i}w_{i+1}\in J^{2}$ for all $x_{i}$.  Thus $\mathbf{m}$ is in the
socle of $S/J^{2}$, and in particular this socle is nonempty, so
$\mathfrak{p}_{\mathbf{m}}=(x_{1},\dots, x_{2\ell -1})$ is associated to
$J^{2}$.  In fact, it is a moderately difficult computation to find an
irredundant primary decomposition:
\begin{proposition}\label{p:oddcycledecomp}
Let $G$ be the odd cycle on $x_{1},\dots,x_{2\ell -1}$.  Then
\[
J^{2}=\left[\bigcap_{i=1}^{2\ell -1} (x_{i},x_{i+1})^{2}\right] \cap
(x_{1}^{2},\dots, x_{2\ell -1}^{2}).
\]
\end{proposition}

\begin{remark}
Proposition \ref{p:oddcycledecomp} picks out the difference between $J^{2}$ and the symbolic square $J^{(2)}$ when $G$ is an odd cycle.  The product of the variables $\mathbf{m}$ appears in $\mathfrak{p}^{2}$ for all $\mathfrak{p}\in \Ass(S/J)$, but is missing from $J^{2}$.  (Combinatorially, this corresponds to $\mathbf{m}$ being a double cover of $G$ that cannot be partitioned into two single covers.)  Thus $\mathbf{m}\in J^{(2)}\smallsetminus J^{2}$.  
\end{remark}

\begin{remark}  We can attempt a similar analysis on an even cycle,
  but we find only two smallest vertex covers,
  $w_{\text{odd}}=x_{1}\dots x_{2\ell-1}$ and
  $w_{\text{even}}=x_{2}\dots x_{2\ell}$.  Then
  $\mathbf{m}=w_{\text{odd}}w_{\text{even}}\in J^{2}$ is not a socle
  element.  In this case Theorem \ref{t:oddcycledecomp} will tell
  us that $J^{2}$ has primary decomposition $\bigcap
  (x_{i},x_{i+1})^{2}$, i.e., $J^{(2)}=J^{2}$.
\end{remark}

In fact, Francisco, H\`a, and Van Tuyl show that, for an arbitrary
graph $G$, the odd cycles can be read off from the associated primes
of $J^{2}$ \cite{FHVT}.  Given a set $W\subset V$, put
$\mathfrak{p}_{W}^{\langle 2\rangle }=(x_{i}^{2}:x_{i}\in W)$.  Then we have:
\begin{theorem}\label{t:oddcycledecomp}  Let
  $G$ be a graph.  Then
  $J^{2}$ has irredundant primary decomposition
\[
J^{2}=\left[\bigcap_{e\in E(G)}\mathfrak{p}_{e}^{2}\right] \cap
\left[\bigcap_{G_{W}\text{ is an induced odd
      cycle}}\mathfrak{p}_{W}^{\langle 2\rangle}\right].
\]
\end{theorem}

\begin{corollary}\label{t:oddcycleass} Let $G$ be a graph.  Then we
  have
\[
\Ass(S/J^{2})=\left\{\mathfrak{p}_{e}:e\in E(G)\right\} \cup
\left\{\mathfrak{p}_{W}:G_{W} \text{ is an induced odd cycle}\right\}.
\]
\end{corollary}

Corollary \ref{t:oddcycleass} and Theorem \ref{simisulrich} are also connected via work of Sturmfels and Sullivant \cite{SS}, who show that generalized Alexander duality connects the secant powers of an ideal with the powers of its dual.

\begin{runningexample}
We have $\Ass(S/J^{2})=E(G)\cup\{(a,b,c,d,e)\}$.  The prime $(a,b,c,d,e)$ appears here because $abcde$ is an odd induced cycle of $G$.  The even induced cycle $defg$ does not appear in $\Ass(S/J^{2})$, nor does the odd circuit $abcdgfe$, which is not induced.  Furthermore, per Theorem \ref{simisulrich}, $I^{\{2\}}$ is generated by the odd cycle $abcde$.
\end{runningexample}

Theorem \ref{t:oddcycledecomp} and Corollary \ref{t:oddcycleass} tell us that the odd cycles of a graph $G$ exactly
describe the difference between the symbolic square and ordinary square
of its cover ideal $J(G)$.  It is natural to ask about hypergraph-theoretic interpretations of
the differences between higher symbolic and ordinary
powers of $J(G)$, and of the differences between these powers for the edge ideal $I(G)$.  The answer to the former question involves \emph{critical hypergraphs},
discussed in \S\ref{s:higherpowers}.
The
latter question is closely related to a problem in combinatorial
optimization theory.  We describe this relationship in \S \ref{s:packing}.


The importance of detecting odd induced cycles in a graph is apparent
in the Strong Perfect Graph Theorem, proven by Chudnovsky, Robertson,
Seymour, and Thomas in \cite{CRST} after the conjecture had been open for
over 40 years. A graph $G$ is \emph{perfect} if for each induced
subgraph $H$ of $G$, the
chromatic number $\chi(H)$ equals the clique number $\omega(H)$, where
$\omega(H)$ is the number of
vertices in the largest clique (i.e., complete subgraph) appearing in
$H$. Perfect graphs are an especially important class of graphs, and
they have a relatively simple characterization. Call any odd cycle of
at least five vertices an \emph{odd hole}, and define an \emph{odd
  antihole} to be the complement of an odd hole. 

\begin{theorem}[Strong Perfect Graph Theorem] \label{SPGT}
A graph is perfect if and only if it contains no odd holes or odd antiholes.
\end{theorem}

Let $G$ be a graph with complementary graph $G^c$ (that is, $G^c$ has the same vertex set as $G$ but the complementary set of edges). Let $J(G)$ be the cover ideal of $G$ and $J(G^c)$ be the cover ideal of $G^c$. Using the Strong Perfect Graph Theorem along with Corollary~\ref{t:oddcycleass}, we conclude that a graph $G$ is perfect if and only if neither $S/J(G)^2$ nor $S/J(G^c)^2$ has an associated prime of height larger than three. It is clear from the induced pentagon that the graph from Running Example~\ref{running} is imperfect; this is apparent algebraically from the fact that $(a,b,c,d,e)$ is associated to $R/J(G)^2$.


\section{Associated primes and perfect graphs} \label{s:higherpowers}

Theorem \ref{t:oddcycledecomp} and Corollary \ref{t:oddcycleass} exhibit a
strong interplay between coloring properties of a graph and associated primes of
the square of its cover ideal. In this section, we explore the connection
between coloring properties of hypergraphs in general and associated primes of
higher powers of their cover ideals. We also specialize back to graphs and give
algebraic characterizations of perfect graphs.

\begin{definition}A \emph{critically $d$-chromatic hypergraph} is a hypergraph
  $G$ with $\chi(G) =d$ whose proper induced subgraphs all have
  smaller chromatic number; $G$ is also called a \emph{critical
    hypergraph}.
\end{definition}

The connection between critical hypergraphs and associated
primes begins with a theorem of Sturmfels and Sullivant on graphs that
generalizes naturally to hypergraphs. 

\begin{theorem} \label{ss-secant}
Let $G$ be a hypergraph with edge ideal $I$. Then the squarefree minimal generators of $I^{\{s\}}$ are the monomials $W$ such that $G_W$ is critically $(s+1)$-chromatic.
\end{theorem}

Higher powers of the cover ideal $J = J(G)$ of a hypergraph have more complicated structure than the square.
It is known that the primes associated to $S/J^2$ persist as associated primes of all $S/J^s$ for $s \ge 2$ \cite[Corollary 4.7]{FHVTperfect}. As one might expect from the case of $J^{2}$, if $H$ is a critically $(d+1)$-chromatic induced subhypergraph of $G$, then $\mathfrak{p}_H \in \Ass(S/J^{d})$ but $\mathfrak{p}_H \notin \Ass(S/J^{e})$ for any $e < d$. However, the following example from \cite{FHVTperfect} illustrates that other associated primes may arise as well.

\begin{example} \label{newkindofprimes}
Let $G$ be the graph with vertices $\{x_1,\dots,x_6\}$ and
edges \[x_1x_2,x_2x_3,x_3x_4,x_4x_5,x_5x_1,x_3x_6,x_4x_6,x_5x_6,\]
where we have abused notation by writing edges as monomials. Thus $G$
is a five-cycle on $\{x_1,\dots,x_5\}$ with an extra vertex $x_6$
joined to $x_3$, $x_4$, and $x_5$. Let $J$ be the cover ideal of
$G$. The maximal ideal $\mathfrak{m}=(x_1,\dots,x_6)$ is associated to
$S/J^3$ but to neither $S/J$ nor $S/J^2$. However, $G$ is not a critically $4$-chromatic graph; instead, $\chi(G)=3$.
\end{example}

Consequently, the critical induced subhypergraphs of a hypergraph $G$
may not detect all associated primes of $S/J^s$. Fortunately, there is
a related hypergraph whose critical induced subhypergraphs do yield a complete list of associated primes. We define the expansion of a hypergraph, the crucial tool.

\begin{definition}\label{d.expansion}
Let $G$ be a hypergraph with vertices $V=\{x_1,\dots,x_n\}$ and edges
$E$, and let $s$ be a positive integer. We create a new hypergraph
$G^s$, called \emph{the $s^{\text{th}}$ expansion of $G$}, as follows. 
We create vertex sets $V_{1}=\{x_{1,1},\dots, x_{n,1}\}$, \dots, $V_{s}=\{x_{1,s},\dots, x_{n,s}\}$.  (We think of these vertex sets as having distinct flavors.  In the literature, the different flavors $x_{i,j}$ of a vertex $x_{i}$ are sometimes referred to as its \emph{shadows}.)  The edges of $G^{s}$ consist of all
edges $x_{i,j}x_{i,k}$ connecting all differently flavored versions of the same vertex, and all edges arising from possible assignments of flavors to the vertices in an edge of $G$.

We refer to the map sending all flavors $x_{i,j}$ of a vertex $x_{i}$ back to $x_{i}$ as \emph{depolarization}, by analogy with the algebraic process of polarization. 
\end{definition}


\begin{example} \label{e.five}
Consider a five-cycle $G$ with vertices $x_1,\dots,x_5$. Then $G^2$ has vertex set $\{x_{1,1}, x_{1,2}, \dots, x_{5,1}, x_{5,2}\}$. Its edge set consists of edges $x_{1,1}x_{1,2}, \dots, x_{5,1}x_{5,2}$ as well as all edges $x_{i,j}x_{i+1,j'}$, where $1 \le j \le j' \le 2$, and the first index is taken modulo 5. Thus, for example, the edge $x_1x_2$ of $G$ yields the four edges $x_{1,1}x_{2,1}$, $x_{1,1}x_{2,2}$, $x_{1,2}x_{2,1}$, and $x_{1,2}x_{2,2}$ in $G^2$.
\begin{figure}[hbtf]
\includegraphics[height=1.5in]{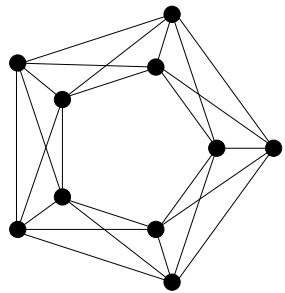}
\caption{The second expansion graph of a 5-cycle}\label{f:expansion}
\end{figure}
\end{example}

Our goal is to understand the minimal monomial generators of the generalized Alexander dual
$(J(G)^{s})^{[\mathbf{s}]}$, where $\mathbf{s}$ is the vector
$(s,\dots,s)$, one entry for each vertex of $G$. Under generalized
Alexander duality, these correspond to the ideals in an irredundant
irreducible decomposition of $J(G)^s$, yielding the associated primes of
$S/J(G)^s$. 

By generalized Alexander duality, Theorem~\ref{ss-secant} identifies the squarefree minimal monomial generators of $(J(G)^s)^{[\mathbf{s}]}$. Understanding the remaining monomial generators requires the following theorem \cite[Theorem 4.4]{FHVTperfect}. For a set of vertices $T$, write $\mathbf{m}_T$ to denote the product of the corresponding variables.

\begin{theorem}\label{t.expansion}
Let $G$ be a hypergraph with cover ideal $J = J(G)$, and let $s$ be a positive integer. Then \[(J^s)^{[\mathbf{s}]} = (\overline{\mathbf{m}_T} ~\big|~ \chi(G^s_T) > s )\]
where $\overline{\bm_T}$ is the depolarization of $\mathbf{m}_T$.
\end{theorem}

The proof relies on a (hyper)graph-theoretic characterization of the generators of $I(G^s)^{\{s\}}$ from Theorem~\ref{ss-secant}. One then needs to prove that $(J^s)^{[\mathbf{s}]}$ is the depolarization of $I(G^s)^{\{s\}}$, which requires some effort; see \cite{FHVTperfect}.

Using Theorem~\ref{t.expansion}, we can identify all associated primes of $S/J(G)^s$ in terms of the expansion graph of $G$.

\begin{corollary} \label{c.allprimes}
Let $G$ be a hypergraph with cover ideal $J = J(G)$. Then $P=(x_{i_1},\dots,x_{i_r}) \in \Ass(S/J^s)$ if and only if there is a subset $T$ of the vertices of $G^s$ such that $G^s_T$ is critically $(s+1)$-chromatic, and $T$ contains at least one flavor of each variable in $P$ but no flavors of other variables.
\end{corollary}

We outline the rough idea of the proof. If $P \in \Ass(S/J^s)$, then
$(x_{i_1}^{e_{i_1}},\dots,x_{i_r}^{e_{i_r}})$ is an 
irreducible component of $J^s$, for some $e_{i_j} > 0$. This yields a
corresponding minimal generator of $(J^s)^{[\mathbf{s}]}$, which gives
a subset $W$ of the vertices of $G^s$ such that $G^s_W$ is critically
$(s+1)$-chromatic, and 
$W$ depolarizes to $x_{i_1}^{e_{i_{1}}} \dots x_{i_r}^{e_{i_{r}}}$. 
Conversely, given a critically $(s+1)$-chromatic expansion hypergraph $G^s_T$, we get a minimal generator of $(J^s)^{[\mathbf{s}]}$ of the form $x_{i_1}^{e_{i_1}} \cdots x_{i_r}^{e_{i_r}}$, where $1 \le e_{i_j} \le s$ for all $i_j$. Duality produces an irreducible component of $J^s$ with radical $P$.

Corollary~\ref{c.allprimes} explains why $\mathfrak{m} \in \Ass(S/J^3)$ in Example~\ref{newkindofprimes}. Let $T$ be the set of vertices \[T = \{x_{1,1},x_{2,1},x_{2,2},x_{3,1},x_{4,1},x_{5,1},x_{6,1}\},\] a subset of the vertices of $G^3$. Then $G^3_T$ is critically 4-chromatic.

As a consequence of this work, after specializing to graphs, we get two algebraic characterizations of perfect graphs that are independent of the Strong Perfect Graph Theorem. First, we define a property that few ideals satisfy (see, e.g., \cite{HostenThomas}).

\begin{definition}\label{d:saturated}
An ideal $I \subset S$ has the \emph{saturated chain property for associated primes}
if given any associated prime $P$ of $S/I$ that is not minimal, there exists an associated prime
$Q \subsetneq P$ with height$(Q)= $ height$(P)-1$.
\end{definition}

We can now characterize perfect graphs algebraically in two different ways \cite[Theorem 5.9]{FHVTperfect}. The key point is that for perfect graphs, the associated primes of powers of the cover ideal correspond exactly to the cliques in the graph.

\begin{theorem} \label{perfectgraphs}
Let $G$ be a simple graph with cover ideal $J$. Then the following are equivalent:
\begin{enumerate}
\item $G$ is perfect.
\item For all $s$ with $1 \le s < \chi(G)$, $P=(x_{i_1}, \dots, x_{i_r}) \in \Ass(R/J^s)$ if and only if
the induced graph on $\{x_{i_1},\dots,x_{i_r}\}$ is a clique of size $1 < r \le s+1$ in $G$.
\item For all $s \geq 1$, $J^s$ has the saturated chain property for associated primes.
\end{enumerate}
\end{theorem}

\begin{proof}
We sketch (1) implies (2) to give an idea of how expansion is used. Suppose $G$ is a perfect graph. A standard result in graph theory shows that $G^s$ is also perfect. Let $P \in \Ass(S/J^s)$, so $P$ corresponds to some subset $T$ of the vertices of $G^s$ such that $G^s_T$ is critically $(s+1)$-chromatic. Because $G^s$ is perfect, the clique number of $G^s_T$ is also $s+1$, meaning there exists a subset $T'$ of $T$ such that $G^s_{T'}$ is a clique with $s+1$ vertices. Thus $G^s_{T'}$ is also a critically $(s+1)$-chromatic graph contained inside $G^s_T$, forcing $T=T'$. Hence $G^s_T$ is a clique, and the support of the depolarization of $\bar{\mathbf{m}_T}$ is a clique with at most $s+1$ vertices. Therefore $G_P$ is a clique.
\end{proof}

\begin{remark}
If $J$ is the cover ideal of a perfect graph, its powers satisfy a condition stronger than that of Definition~\ref{d:saturated}. If $P \in \Ass(S/J^s)$, and $Q$ is any monomial prime of height at least two contained in $P$, then $Q \in \Ass(S/J^s)$. This follows from the fact that $P$ corresponds to a clique in the graph.
\end{remark}

Theorem~\ref{perfectgraphs} provides information about two classical issues surrounding associated primes of powers of ideals. Brodmann proved that for any ideal $J$, the set of associated primes of $S/J^s$ stabilizes \cite{Brodmann}. However, there are few good bounds in the literature for the power at which this stabilization occurs. When $J$ is the cover ideal of a perfect graph, Theorem~\ref{perfectgraphs} demonstrates that stabilization occurs at $\chi(G)-1$. Moreover, though in general associated primes may disappear and reappear as the power on $J$ increases (see, e.g., \cite{BHH,HH} and also \cite[Example 4.18]{MV}), when $J$ is the cover ideal of a perfect graph, we have $\Ass(S/J^s) \subseteq \Ass(S/J^{s+1})$ for all $s \ge 1$. In this case, we say that $J$ has the \emph{persistence property for associated primes}, or simply the \emph{persistence property}. Morey and Villarreal give an alternate proof of the persistence property for cover ideals of perfect graphs in \cite[Example 4.21]{MV}.

While there are examples of arbitrary monomial ideals for which persistence fails, we know of no such examples of \emph{squarefree} monomial ideals. Francisco, H\`a, and Van Tuyl (see \cite{FHVT, FHVT2}) have asked:

\begin{question} \label{sfpersistence}
Suppose $J$ is a squarefree monomial ideal. Is $\Ass(S/J^s) \subseteq \Ass(S/J^{s+1})$ for all $s \ge 1$?
\end{question}

While Question~\ref{sfpersistence} has a positive answer when $J$ is
the cover ideal of a perfect graph, little is known for cover ideals
of imperfect graphs. Francisco, H\`a, and Van Tuyl answer
Question~\ref{sfpersistence} affirmatively for odd holes and odd
antiholes in \cite{FHVT2}, but we are not aware of any other imperfect
graphs whose cover ideals are known to have this persistence property. One possible approach is to exploit the machinery of expansion again. Let $G$ be a graph, and let $x_i$ be a vertex of $G$. Form the expansion of $G$ at $\{x_i\}$ by replacing $x_i$ with two vertices $x_{i,1}$ and $x_{i,2}$, joining them with an edge. For each edge $\{v,x_i\}$ of $G$, create edges $\{v,x_{i,1}\}$ and $\{v,x_{i,2}\}$. If $W$ is any subset of the vertices of $G$, form $G[W]$ by expanding all the vertices of $W$. Francisco, H\`a, and Van Tuyl conjecture:

\begin{conjecture} \label{c.expansion}
Let $G$ be a graph that is critically $s$-chromatic. Then there exists a subset $W$ of the vertices of $G$ such that $G[W]$ is critically $(s+1)$-chromatic.
\end{conjecture}

In \cite{FHVT2}, Francisco, H\`a, and Van Tuyl prove that if
Conjecture~\ref{c.expansion} is true for all $s \ge 1$, then all cover
ideals of graphs have the persistence property. One can also state a
hypergraph version of Conjecture~\ref{c.expansion}; if true, it
would imply persistence of associated primes for all squarefree
monomial ideals.

Finally, in \cite{MV}, Morey and Villarreal prove persistence for edge
ideals $I$ of any graphs containing a leaf (a vertex of degree 1). Their proof passes to the associated graded ring, and the vital step is identifying a regular element of the associated graded ring in $I/I^2$. Morey and Villarreal remark that attempts to prove persistence results for more general squarefree monomial ideals lead naturally to questions related to the Conforti-Cornu\'ejols conjecture, discussed in the following section.


\section{Equality of symbolic and ordinary powers and linear programming} \label{s:packing}

We have seen in the last section that comparing symbolic and ordinary powers of the cover ideal of a hypergraph allows us to study structures and coloring properties of the hypergraph. In this section, we address the question of when symbolic and ordinary powers of a squarefree monomial ideal are the same, and explore an algebraic approach to a long-standing conjecture in linear integer programming, the Conforti-Cornu\'ejols conjecture. 
In what follows, we state the Conforti-Cornu\'ejols conjecture in its original form,
describe how to translate the conjecture into algebraic language, and discuss
its algebraic reformulation and related problems.

The Conforti-Cornu\'ejols conjecture states the equivalence between the packing
and the max-flow-min-cut properties for \emph{clutters} which, as noted before, are essentially simple
hypergraphs. 

As before, $G = (V,E)$ denotes a hypergraph with $n$ vertices $V =
\{x_1, \dots, x_n\}$ and $m$ edges $E = \{e_1, \dots, e_m\}$. Let $A$ 
be the \emph{incidence matrix} of $G$, i.e., the $(i,j)$-entry
of $A$ is 1 if the vertex $x_i$ belongs to the edge $e_j$, and 0 otherwise. For
a nonnegative integral vector $\c \in \ZZ_{\ge 0}^n$, consider
the following
dual linear programming system
\begin{align}
\max \{\langle \1, \y\rangle ~|~ \y \in \RR^m_{\ge 0}, A\y \le \c\} = \min \{\langle \c, \z\rangle ~|~ \z \in \RR^n_{\ge 0}, A^\text{T}\z \ge \1\}. \label{eq.CCdualsystem}
\end{align}

\begin{definition} \label{def.properties}
Let $G$ be a simple hypergraph.
\begin{enumerate}
\item The hypergraph $G$ is said to \emph{pack} if the dual system (\ref{eq.CCdualsystem}) has integral optimal solutions $\y$ and $\z$ when $\c = \1$.
\item The hypergraph $G$ is said to have the \emph{packing property} if the dual system (\ref{eq.CCdualsystem}) has integral optimal solutions $\y$ and $\z$ for all vectors $\c$ with components equal to 0, 1 and $+\infty$.
\item The hypergraph $G$ is said to have the \emph{max-flow-min-cut (MFMC)
property} or to be \emph{Mengerian} if the dual system (\ref{eq.CCdualsystem})
has integral optimal solutions $\y$ and $\z$ for all nonnegative integral
vectors $\c \in \ZZ_{\ge 0}^n$.
\end{enumerate}
\end{definition}

\begin{remark}
In Definition \ref{def.properties}, setting an entry of $\c$ to $+\infty$ means that this entry is sufficiently large, so the corresponding inequality in
the system $A\y \le \c$ can be omitted. It is clear that if $G$
satisfies the MFMC property, then it has the packing property. 
\end{remark}

The following conjecture was stated in \cite[Conjecture 1.6]{Cornuejols} with a reward prize of \$5,000 for the solution.

\begin{conjecture}[Conforti-Cornu\'ejols] \label{conj.CC}
A hypergraph has the packing property if and only if it has the max-flow-min-cut property.
\end{conjecture}

As we have remarked, the main point of Conjecture \ref{conj.CC} is to show that
if a hypergraph has the packing property then it also has the MFMC property.

The packing property can be understood via more familiar concepts in
(hyper)graph theory, namely \emph{vertex covers} (also referred to as
\emph{transversals}), which we recall from Section~\ref{intro}, and
\emph{matchings}.

\begin{definition}
A \emph{matching} (or \emph{independent set}) of a hypergraph $G$ is a set of pairwise disjoint edges.
\end{definition}

Let $\alpha_0(G)$ and $\beta_1(G)$ denote the minimum cardinality of
a vertex cover and the maximum cardinality of a matching in $G$,
respectively. We have
$\alpha_0(G) \ge \beta_1(G)$ since every edge in any matching must
hit at least one vertex from every cover.

The hypergraph $G$ is said to be \emph{K\"{o}nig} if
$\alpha_0(G)=\beta_1(G)$. Observe that giving a vertex cover and a
matching of equal size for $G$ can be viewed as giving integral solutions to the dual system (\ref{eq.CCdualsystem}) when $\c = \1$. Thus, $G$ is K\"onig if and only if $G$ packs.

There are two operations commonly used on a hypergraph $G$ to
produce new, related hypergraphs on smaller vertex sets. Let $x \in V$ be a vertex in $G$.
The \emph{deletion} $G \setminus x$ is formed by removing $x$ from the vertex set and deleting any edge in $G$ that
contains $x$. The \emph{contraction} $G / x$ is obtained by removing $x$ from the vertex set and removing $x$ from any edge of $G$ that contains $x$. Any hypergraph obtained from $G$ by a sequence of deletions and contractions is called a \emph{minor} of $G$. Observe that the deletion and contraction of a vertex $x$ in $G$ has the same effect as setting the corresponding component in $\c$ to $+\infty$ and 0, respectively, in the dual system (\ref{eq.CCdualsystem}). Hence,
\begin{center}
\emph{$G$ satisfies the packing property if and only if $G$ and all of its minors are K\"onig.}
\end{center}

\begin{example} Let $G$ be a 5-cycle. Then $G$ itself is not K\"onig ($\alpha_0(G) = 3$ and $\beta_1(G) = 2$). Thus, $G$ is does not satisfy the packing property.
\end{example}

\begin{example} Any bipartite graph is K\"onig. Therefore, if $G$ is a bipartite graph then (since all its minors are also bipartite) $G$ satisfies the packing property.
\end{example}

We shall now explore how Conjecture \ref{conj.CC} can be understood via commutative algebra, and more specifically, via algebraic properties of edge ideals.

As noted in Section \ref{s:prel}, symbolic Rees algebras are more complicated than the ordinary Rees algebras, and could be non-Noetherian. Fortunately, in our situation, the symbolic Rees algebra of a squarefree monomial ideal is always Noetherian and finitely generated (cf. \cite[Theorem 3.2]{HHT}). Moreover, the symbolic Rees algebra of the edge ideal of a hypergraph $G$ can also be viewed as the \emph{vertex cover algebra} of the dual hypergraph $G^*$.

\begin{definition} Let $G = (V,E)$ be a simple hypergraph over the vertex set $V = \{x_1, \dots, x_n\}$.
\begin{enumerate}
\item We call a nonnegative integral vector $\c = (c_1, \dots, c_n)$ a \emph{$k$-cover} of $G$ if $\sum_{x_i \in e} c_i \ge k$ for any edge $e$ in $G$.
\item The \emph{vertex cover algebra} of $G$, denoted by $\A(G)$, is defined to be
$$\A(G) = \bigoplus_{k \ge 0} \A_k(G),$$
where $\A_k(G)$ is the $k$-vector space generated by all monomials $x_1^{c_1} \dots x_n^{c_n}t^k$ such that $(c_1, \dots, c_n) \in \ZZ^n_{\ge 0}$ is a $k$-cover of $G$.
\end{enumerate}
\end{definition}

\begin{lemma} \label{lem.vertexcoveralgebra}
Let $G$ be a simple hypergraph with edge ideal $I = I(G)$, and let $G^*$ be its dual hypergraph. Then
$$\R_s(I) = \A(G^*).$$
\end{lemma}

We are now ready to give an algebraic interpretation of the MFMC property.

\begin{lemma} \label{lem.sigmagamma}
Let $G = (V,E)$ be a simple hypergraph with $n$ vertices and $m$ edges. Let $A$ be its incidence matrix. For a nonnegative integral vector $\c \in \ZZ^n_{\ge 0}$, define \\
\hspace*{6ex} $\sigma(\c) = \max \{ \langle \1, \y \rangle ~|~ \y \in \ZZ^m_{\ge 0}, A\y \le \c \}$, and \\
\hspace*{6ex} $\gamma(\c) = \min \{ \langle \c, \z \rangle ~|~ \z \in \ZZ^n_{\ge 0}, A^{\text{T}}\z \ge \1 \}.$  \\
Then
\begin{enumerate}
\item $\c$ is a $k$-cover of $G^*$ if and only if $k \le \gamma(\c)$.
\item $\c$ can be written as a sum of $k$ vertex covers of $G^*$ if and only if $k \le \sigma(\c)$. 
\end{enumerate}
\end{lemma}

\begin{proof}
By definition, a nonnegative integral vector $\c = (c_1, \dots, c_n) \in \ZZ^n_{\ge 0}$ is a $k$-cover of $G^*$ if and only if
\begin{align}
k \le \min \{ \sum_{x_i \in e} c_i ~|~ e \text{ is any edge of } G^* \}. \label{eq.k-cover}
\end{align}
Let $\z$ be the $(0,1)$-vector representing $e$. Observe that $e$ is an edge of $G^*$ if and only if $e$ is a minimal vertex cover of $G$, and this is the case if and only if $A^{\text{T}}\z \ge \1$. Therefore, the condition in (\ref{eq.k-cover}) can be translated to
\begin{align*}
k & \le \min \{ \langle \c, \z \rangle ~|~ \z \in \{0,1\}^n, A^{\text{T}}\z \ge \1 \} \\
& = \min \{ \langle \c, \z \rangle ~|~ \z \in \ZZ^n_{\ge 0}, A^{\text{T}}\z \ge \1 \} = \gamma(\c).
\end{align*}

To prove (2), let $\mba_1, \dots, \mba_m$ be representing vectors of the edges in $G$ (i.e., the columns of the incidence matrix $A$ of $G$). By Proposition~\ref{alexduality}, $\mba_1, \dots, \mba_m$ represent the minimal vertex cover of the dual hypergraph $G^*$. One can show that a nonnegative integral vector $\c \in \ZZ^n$ can be written as the sum of $k$ vertex covers (not necessarily minimal) of $G^*$ if and only if there exist integers $y_1, \dots, y_m \ge 0$ such that $k = y_1 + \dots + y_m$ and $y_1\mba_1 + \dots + y_m\mba_m \le \c$. Let $\y = (y_1, \dots, y_m)$. Then
$$\langle \1, \y \rangle = y_1 + \dots + y_m \text{ and } A\y = y_1\mba_1 + \dots + y_m\mba_m.$$
Thus,
$$\sigma(\c) = \max \{ k ~|~ \c \text{ can be written as a sum of } k \text{ vertex covers of } G^* \}.$$
\end{proof}

\begin{theorem} \label{thm.MFMCreduce}
Let $G$ be a simple hypergraph with dual hypergraph $G^*$. Then the dual linear programming system (\ref{eq.CCdualsystem}) has integral optimal solutions $\y$ and $\z$ for all nonnegative integral vectors $\c$ if and only if $\R_s(I(G)) = \A(G^*)$ is a standard graded algebra; or equivalently, if and only if $I(G)^{(q)} = I(G)^q$ for all $q \ge 0$.
\end{theorem}

\begin{proof} Given integral optimal solutions $\y$ and $\z$ of the dual system (\ref{eq.CCdualsystem}) for a nonnegative integral vector $\c$, we get
$$\sigma(\c) = \gamma(\c).$$
The conclusion then follows from Lemmas \ref{lem.vertexcoveralgebra} and
\ref{lem.sigmagamma}.
\end{proof}

The following result (see \cite[Corollary 1.6]{HHTZ} and \cite[Corollary 3.14]{GVV}) gives an algebraic approach to Conjecture \ref{conj.CC}.

\begin{theorem} \label{thm.33}
Let $G$ be a simple hypergraph with edge ideal $I = I(G)$. The following conditions are equivalent:
\begin{enumerate}
\item $G$ satisfies the MFMC property,
\item $I^{(q)} = I^q$ for all $q \ge 0$,
\item The associated graded ring $\gr_{I} := \bigoplus_{q \ge 0} I^q/I^{q+1}$ is reduced,
\item $I$ is \emph{normally torsion-free}, i.e., all powers of $I$ have the same associated primes.
\end{enumerate}
\end{theorem}

\begin{proof} The equivalence between (1) and (2) is the content of Theorem \ref{thm.MFMCreduce}. The equivalences of (2), (3) and (4) are well known results in commutative algebra (cf. \cite{HSV}).
\end{proof}

The Conforti-Cornu\'ejols conjecture now can be restated as follows.

\begin{conjecture} \label{conj.CCalgebraic}
Let $G$ be a simple hypergraph with edge ideal $I = I(G)$. If $G$ has packing property then the associated graded ring $\gr_I$ is reduced. Equivalently, if $G$ and all its minors are K\"onig, then the associated graded ring $\gr_I$ is reduced.
\end{conjecture}

It remains to give an algebraic characterization for the packing
property. To achieve this, we shall need to interpret minors and the K\"onig property.
Observe that the deletion $G \setminus x$ at a vertex $x \in \X$ has the effect of setting
$x = 0$ in $I(G)$ (or equivalently, of passing to the ideal $(I(G),x)/(x)$ in
the quotient ring $S/(x)$), and the contraction $G / x$ has the effect of
setting $x=1$ in $I(G)$ (or equivalently, of passing to the ideal $I(G)_x$ in
the localization $S_x$). Thus, we call an ideal $I'$ a \emph{minor} of a
squarefree monomial ideal $I$ if $I'$ can be obtained from $I$ by a sequence of
taking quotients and localizations at the variables.
Observe further that $\alpha_0(G) = \height I(G)$, and
if we let $\operatorname{m-grade} I$ denote the maximum length of a regular
sequence of monomials in $I$ then $\beta_1(G) = \operatorname{m-grade} I(G)$.
Hence, a simple hypergraph with edge ideal $I$ is K\"onig if $\height I = \operatorname{m-grade} I$.
This leads us to a complete algebraic reformulation of the
Conforti-Cornu\'ejols conjecture:

\begin{conjecture} \label{conj.trans}
    Let $I$ be a squarefree monomial ideal such that $I$ and all of its minors satisfy the property that their heights are the same as their m-grades. Then $\gr_I$ is reduced; or equivalently, $I$ is normally torsion-free.
\end{conjecture}

The algebraic consequence of the conclusion of Conjecture \ref{conj.trans}
(and equivalently, Conjecture \ref{conj.CC}) is the equality $I^{(q)} =
I^q$ for all $q \ge 0$ or, equivalently, the normally torsion-freeness of $I$.
If one is to consider the equality $I^{(q)} = I^q$, then it is
natural to look for an integer $l$ such that $I^{(q)} = I^q$ for $0 \le q \le
l$ implies $I^{(q)} = I^q$ for all $q \ge 0$, or to examine squarefree monomial
ideals with the property that $I^{(q)} = I^q$ for all $q \ge q_0$. On the
other hand, if one is to
investigate the normally torsion-freeness then it is natural to study
properties of minimally not normally torsion-free ideals. The following problem
is naturally connected to Conjectures \ref{conj.CC} and \ref{conj.trans},
and part of it has been the subject of work in commutative algebra (cf.
\cite{HS}).

\begin{problem} \label{prob.algebraic}
Let $I$ be a squarefree monomial ideal in $S = K[x_1, \dots, x_n]$.
\begin{enumerate}
\item Find the least integer $l$ (may depend on $I$) such that if $I^{(q)} =
I^q$ for $0 \le q \le l$ then $I^{(q)} = I^q$ for all $q \ge 0$.
\item Suppose that there exists a positive integer $q_0$ such that $I^{(q)} =
I^q$ for all $q \ge q_0$. Study algebraic and combinatorial properties of $I$.
\item Suppose $I$ is minimally not normally torsion-free (i.e., $I$ is not
normally torsion-free but all its minors are). Find the least power $q$ such
that $\Ass(S/I^q) \not= \Ass(S/I)$.
\end{enumerate}
\end{problem}


\begin{thebibliography}{9999999}

\bibitem[BHH]{BHH} S. Bandari, J. Herzog, and T. Hibi, Monomial ideals whose depth function has any given number of strict local maxima. Preprint, 2012. {\tt arXiv:1205.1348}

\bibitem[BoHa]{BH} C. Bocci and B. Harbourne, Comparing powers and symbolic powers of ideals. \emph{J. Algebraic Geom.} {\bf 19} (2010), no. 3, 399--417.

\bibitem[Br]{Brodmann} M. Brodmann, Asymptotic stability of ${\rm Ass}(M/I\sp{n}M)$.
\emph{Proc. Amer. Math. Soc.} \textbf{74} (1979), 16--18.

\bibitem[BrHe]{Bruns-Herzog} W. Bruns and J. Herzog, Cohen-Macaulay rings. Cambridge Studies in Advanced Mathematics, {\bf 39}. \emph{Cambridge University Press, Cambridge}, 1993.

\bibitem[CGG]{CGG} M. V. Catalisano, A. V. Geramita, and A. Gimigliano, On the ideals of secant varieties to certain rational varieties. \emph{J. Algebra} {\bf 319} (2008), no. 5, 1913–-1931.

\bibitem[CRST]{CRST} M. Chudnovsky, N. Robertson, P. Seymour, and R. Thomas,
The strong perfect graph theorem. \emph{Ann. of Math.} (2)  {\bf 164} (2006), 51--229.

\bibitem[C]{Cornuejols} G. Cornu\'ejols, Combinatorial optimization: Packing and covering. CBMS-NSF Regional Conference Series in Applied Mathematics, {\bf 74}, \emph{SIAM, Philadelphia}, 2001.

\bibitem[ELS]{ELS} L. Ein, R. Lazarsfeld, and K.E. Smith, Uniform bounds and symbolic powers on smooth varieties. \emph{Invent. Math.} {\bf 144} (2001), no. 2, 241--252.

\bibitem[FHVT1]{FHVT} C. A. Francisco, H. T. H\`a, and A. Van Tuyl, Associated primes of monomial ideals and odd holes in graphs. \emph{J. Algebraic Combin.} {\bf 32} (2010), no. 2, 287--301.

\bibitem[FHVT2]{FHVT2} C. A. Francisco, H. T. H\`a, and A. Van Tuyl, A conjecture on critical graphs and connections to the persistence of associated primes.  \emph{Discrete Math.} {\bf 310} (2010), no. 15--16, 2176--2182.

\bibitem[FHVT3]{FHVTperfect} C. A. Francisco, H. T. H\`a, and A. Van Tuyl, Colorings of hypergraphs, perfect graphs, and associated primes of powers of monomial ideals. \emph{J. Algebra} {\bf 331} (2011), 224–-242.

\bibitem[GVV]{GVV} I. Gitler, C.E. Valencia, and R. Villarreal, A note on Rees algebras and the MFMC property. \emph{Beitr\"age Algebra Geom.} {\bf 48} (2007), no. 1, 141--150.

\bibitem[HS]{HS} H.T. H\`a, S. Morey, Embedded associated primes of powers of square-free monomial ideals. \emph{J. Pure Appl. Algebra} {\bf 214} (2010), no. 4, 301--308.

\bibitem[HeHi]{HH} J. Herzog, T. Hibi, The depth of powers of an ideal. \emph{J. Algebra} {\bf 291} (2005), 534--550.

\bibitem[HHT]{HHT} J. Herzog, T. Hibi, and N.V. Trung, Symbolic powers of monomial ideals and vertex cover algebras. \emph{Adv. Math.} {\bf 210} (2007), no. 1, 304--322.

\bibitem[HHTZ]{HHTZ} J. Herzog, T. Hibi, N.V. Trung, and X. Zheng, Standard graded vertex cover algebras, cycles and leaves. \emph{Trans. Amer. Math. Soc.} {\bf 360} (2008), no. 12, 6231--6249.

\bibitem[HoHu]{HoHu} M. Hochster and C. Huneke, Comparison of symbolic and ordinary powers of ideals. \emph{Invent. Math.} {\bf 147} (2002), no. 2, 349--369.

\bibitem[HT]{HostenThomas} S. Hosten and R. Thomas, The associated primes of initial ideals of lattice ideals. \emph{Math. Res. Lett.} \textbf{6} (1999), no. 1, 83--97.

\bibitem[HKV]{HKV} C. Huneke, D. Katz, and J. Validashti, Uniform equivalence of symbolic and adic topologies. \emph{Illinois J. Math.} {\bf 53} (2009), no. 1, 325--338.

\bibitem[HSV]{HSV} C. Huneke, A. Simis, and W. Vasconcelos, Reduced normal cones are domains. Invariant theory (Denton, TX, 1986), 95--101, Contemp. Math. {\bf 88}. \emph{Amer. Math. Soc., Providence, RI}, 1989.

\bibitem[MS]{MS} E. Miller and B. Sturmfels, Combinatorial Commutative Algebra. GTM {\bf 227}, \emph{Springer-Verlag}, 2004.

\bibitem[MT]{MT} N.C. Minh and N.V. Trung, Cohen-Macaulayness of monomial ideals and symbolic powers of Stanley-Reisner ideals. \emph{Adv. Math.} {\bf 226} (2011), no. 2, 1285--1306.

\bibitem[MV]{MV} S. Morey and R.H. Villarreal, Edge ideals: algebraic and combinatorial properties. Progress in Commutative Algebra: Ring Theory, Homology, and Decomposition. \emph{de Gruyter, Berlin}, 2012, 85--126.

\bibitem[N]{Nagata} M. Nagata, On the $14^\text{th}$ problem of Hilbert. \emph{Amer. J. Math.} {\bf 81} (1959), 766--772.

\bibitem[P]{Peeva} I. Peeva, Graded syzygies. Algebra and Applications, {\bf 14}. \emph{Springer-Verlag London, Ltd., London}, 2011.

\bibitem[R]{Roberts} P. Roberts, A prime ideal in a polynomial ring whose symbolic blow-up is not Noetherian. \emph{Proc. Amer. Math. Soc.} {\bf 94} (1985), 589--592.

\bibitem[SU]{SU} A. Simis and B. Ulrich, On the ideal of an embedded join.
\emph{J. Algebra} {\bf 226} (2000), no. 1, 1--14.

\bibitem[SS]{SS}  B. Sturmfels and S. Sullivant, Combinatorial secant varieties.
\emph{Pure Appl. Math. Q.} {\bf 2} (2006),  no. 3, part 1, 867--891.

\bibitem[TeTr]{TTerai} N. Terai and N.V. Trung, Cohen-Macaulayness of large powers of Stanley-Reisner ideals. \emph{Adv. Math.} {\bf 229} (2012), no. 2, 711--730.

\bibitem[TrTu]{TT} N.V. Trung and T.M. Tuan, Equality of ordinary and symbolic powers of Stanley-Reisner ideals. \emph{J. Algebra} {\bf 328} (2011), 77--93.

\bibitem[V]{Varbaro} M. Varbaro, Symbolic powers and matroids. \emph{Proc. Amer. Math. Soc.} {\bf 139} (2011), no. 7, 2357--2366.

\end{thebibliography}
\end{document}